\documentclass[12pt]{article}
\usepackage[fleqn]{amsmath}
\usepackage{amsthm}
\usepackage{amssymb}
\usepackage{mathtools}
\newtheorem{theorem}{Theorem}[section]

\newtheorem{corol}[theorem]{Corollary}
\newtheorem{remark}[theorem]{Remark}

\def\cocoa{\mbox{\rm
   C\kern-.13em o\kern-.07 em C\kern-.13em o\kern-.15em A\ }}

\usepackage{setspace}

\title{\bf Geometric Phase Integrals and
 Irrationality Tests}

\author{Domenico Napoletani \thanks{Institute for Quantum Studies, Chapman University, Orange, CA, 92866. Email:napoleta@chapman.edu}, and Daniele C. Struppa\thanks{Schmid College of Science and Technology, Chapman University, Orange, CA 92866. Email:
struppa@chapman.edu}}

\date{}
\begin{document}
\maketitle
\onehalfspacing
\begin{abstract}
Let $F(x)$ be an analytical, real valued function defined on a compact domain $\mathcal {B}\subset\mathbb{R}$. We prove that the problem of establishing the irrationality of $F(x)$ evaluated at $x_0\in \mathcal{B}$ can be stated with respect to the convergence of the phase of a suitable integral $I(h)$, defined on an open, bounded domain, for $h$ that goes to infinity. This is derived as a consequence of a similar equivalence, that establishes the existence of isolated solutions of systems equations of analytical functions on compact real domains in $\mathbb{R}^p$, if and only if the phase of a suitable ``geometric'' complex phase integral $I(h)$ converges for $h\rightarrow \infty$. We finally highlight how the method can be easily adapted to be relevant for the study of the existence of rational or integer points on curves in bounded domains, and we sketch some potential theoretical developments of the method.

\end{abstract}
\newpage
\noindent

\section{Real Geometry and Irrationality}

Real geometry, and especially real algebraic geometry, has developed relatively late its own techniques \cite{RAG}, perhaps due to the great success of algebraic geometry theories over complex fields. However, this has led to a lack of tight bounds on the structure of solutions of systems of equations over the real numbers, and, even with the extensive recent development of real algebraic geometry and its relations to the theory of computation ( see for example \cite{smale}), a general tool that can encompass problems on a very large class of functions is lacking.

In this paper we show that the existence of isolated solutions of a system of analytical equations ${\bf F}(x)=0$ over a compact $\mathcal B$ in $\mathbb{R}^p$ is equivalent, under suitable conditions, to the existence of the limiting phase of a complex phase integral $I_{\bf F}(h)$ for $h$ that goes to infinity.  This result allows a plethora of analytical techniques for the asymptotic and non-perturbative study of complex phase integrals to become relevant for real geometry, providing, in some sense, an indirect, but tightly tailored complexification of real geometry.

We then approach another, apparently unrelated question: establishing the irrationality of special numbers. Geometry has always been deeply intertwined with the problem of establishing the irrationality of numbers, but, while the more specific problem of establishing linear and algebraic independence of several point evaluations of special functions has a rich modern history \cite{baker}, establishing directly the irrationality of series, and of pointwise evaluation of general functions, has been, mostly, the domain of ad-hoc methods (see for example \cite{aperyNew, erdos}).

In Section 3 we suggest that, for real irrational numbers, another viewpoint is available, that transforms the problem of the irrationality of $F(x_0)=\alpha_0$ into the geometric problem of finding zeros of a systems of equations on a four dimensional open, bounded domain. This problem is then phrased in terms of the phase integral method we developed for real geometry in Section 2, offering a new perspective on some old problems. We conclude the paper by suggesting how the main results of Section 3 can be adapted to diophantine geometry.

\section{Geometric Phase integrals}

Given an analytical vectorial function ${\bf F}(x)=[F_1(x),...,F_p(x)]$ defined on a compact set $\mathcal{B}\subset\mathbb{R}^p$, consider the associated norm function
$L(x)=\sum_{i=1}^p F_i(x)^2$.  Then $L(x)=0$ clearly implies $F_i(x)=0$ for all $i$. Moreover, every point such that $L(x)=0$ is a critical point of $L(x)$, since $\frac{\partial L(x)}{\partial x_i}=\sum_{j=1}^p 2F_j(x)\frac{\partial F_j(x)}{\partial x_i}$ and setting $F_j(x)=0$ for all $j$'s gives $\frac{\partial L(x)}{\partial x_i}=0$. The relations between critical points of $L(x)$ and solutions of the system of equations ${\bf F}(x)=0$ can be made more compelling, by building a suitable phase integral whose asymptotic behavior depends on the existence of solutions to the system itself. Indeed the following theorem holds:
\begin{theorem}
\label{compactCase}
Let ${\bf F}(x)=[F_1(x),...,F_p(x)]$ be an analytical, vectorial function defined on a compact domain $\mathcal{B}\subset \mathbb{R}^p$, and let $L(x)=\sum_{i=1}^p F_i(x)^2$ have only isolated critical points in $\mathcal{B}$. Consider the integral
\begin{equation}
\label{PI1}
I(h)=\int_{\mathcal A} \int_{\mathcal B} e^{ihL(x)y^2} dx dy,\,\,\, y\in \mathcal{A} \subset \mathbb{R},\,\,0\not \in \mathcal{A},\,\, x\in \mathcal{B} \subset \mathbb{R}^p,
\end{equation}
and denote by $\phi(I(h))$ the phase of $I(h)$, then the system ${\bf F}(x)=0$ has a solution in $\mathcal{B}$ if and only if the phase $\phi(I(h))$ has a limit for $h$ going to infinity.
\end{theorem}

\begin{proof}
The integration in $x$ in the integral in Eq. \ref{PI1} can be written, for $h\rightarrow \infty$, with respect to the critical points of $L(x)$ in $\mathcal B$, using standard stationary phase approximation methods \cite{wang,erd}, in this paper we will refer mostly to \cite{wang} for the necessary background material.

We can consider separately the critical points such that $L(x)=0$, and those for which $L(x)\neq 0$,and we have:
\begin{equation}
\label{asymp1}
\begin{split}
\lim_{h\rightarrow \infty} I(h)=\int_{y\in \mathcal{A}} \sum_{L(x_i)=0} (\frac{2\pi}{h})^{\frac{p}{2}} \frac{1}{y^p (\det H(x_i))^{1/2}}e^{i\frac{\pi}{4}\sigma_i} dy+ \\
\int_{y\in \mathcal{A}} \sum_{L(x_j)\neq 0} (\frac{2\pi}{h})^{\frac{p}{2}} \frac{1}{y^p (\det H(x_j))^{1/2}}e^{ihL(x_j)y^2+i\frac{\pi}{4}\sigma_j} dy
\end{split}
\end{equation}
\noindent
with $H(x_*)$ the Hessian matrix of $L(x)$ evaluated at $x_*$, and $\sigma_*$ the signature of $H(x_*)$.

We are assuming here that there is at least one critical point with Hessian different from zero, very similar arguments to those we present here can be deduced without this restriction, at the cost of a more complicated argument that depends on higher derivatives of $L$ (that exist, since $L$ is analytical in $\mathcal B$). The restricted (but generic) setting with at least one critical point with Hessian different from zero is sufficient to prove our main result in Section 3.

Since the function $L$ is analytical, so is the system of equations whose solution defines critical points, and if we assume all such solutions are isolated, they are in finite number over a compact set (see for example \cite{kranz}, page 180). And we can further simplify the representation of $I(h)$ and in particular, in the limit of large $h$,
\begin{equation}
\begin{split}
\lim_{h\rightarrow \infty}I_1(h)=\lim_{h\rightarrow \infty}\int_{y\in \mathcal{A}} \sum_{L(x_i)=0} (\frac{2\pi}{h})^{\frac{p}{2}} \frac{1}{y^p (\det H(x_i))^{1/2}}e^{i\frac{\pi}{4}\sigma_i} dy= \\
\lim_{h\rightarrow \infty}\sum_{L(x_i)=0} (\frac{2\pi}{h})^{\frac{p}{2}} \frac{1}{(\det H(x_i))^{1/2}}e^{i\frac{\pi}{4}\sigma_i}S
\end{split}
\end{equation}
\noindent
where $S=\int_{\mathcal{A}}\frac{1}{y^p}dy$, and, being the last term a finite sum, and factoring out $\frac{1}{h^{p/2}}$, the phase of $I_1(h)$ is shown to be independent of $h$ and dependent only on the critical points $x_i$'s, or more exactly, on the signatures $\sigma_i$.

Let us now analyze the second portion of the sum in the right hand side of Eq. \ref{asymp1}:
\begin{equation}
I_2(h)=\int_{y\in \mathcal{A}} \sum_{L(x_j)\neq 0} (\frac{2\pi}{h})^{\frac{p}{2}} \frac{1}{y^p (\det H(x_j))^{1/2}}e^{ihL(x_j)y^2+i\frac{\pi}{4}\sigma_j} dy
\end{equation}

We note first of all that each integral
\begin{equation}
\int_{y\in \mathcal{A}}  (\frac{2\pi}{h})^{\frac{p}{2}} \frac{1}{y^p (\det H(x_j))^{1/2}}e^{ihL(x_j)y^2+i\frac{\pi}{4}\sigma_j} dy
\end{equation}
can be written as
\begin{equation}
(\frac{2\pi}{h})^{\frac{p}{2}} \frac{1}{(\det H(x_j))^{1/2}}e^{i\frac{\pi}{4}\sigma_j}\int_{y\in \mathcal{A}}\frac{1}{y^p} e^{ihL(x_j)y^2} dy
\end{equation}
and it is therefore a phase integral in $y$ computed over an interval that does not include a critical point ($y=0$). Such integral decreases at least like $O(\frac{1}{hL(x_j)})$, the leading contribution from the boundary points of $\mathcal{A}$ (\cite{wang}, page 488; \cite{erd}, page 52). And therefore
\begin{equation}
\lim_{h\rightarrow \infty}I_2(h)=\lim_{h\rightarrow \infty} (2\pi)^{\frac{p}{2}}\sum_{L(x_j)\neq 0}  \frac{1}{(\det H(x_j))^{1/2}}e^{i\frac{\pi}{4}\sigma_j}O(\frac{1}{h^{p/2+1}L(x_j)})
\end{equation}

Recall we are in the generic case where there is at least one point $x_j$ with $\det H(x_j)\neq 0$, and we know there are finitely many critical points, and therefore also finitely many critical points for which $L(x_j)\neq 0$. This last observation allows us to conclude that all the values $L(x_j)$ can be bounded away from $0$, and the entire sum above can be estimated as
\begin{equation}
\lim_{h\rightarrow \infty}I_2(h)= O(\frac{1}{h^{p/2+1}})
\end{equation}
This is a negligible quantity with respect to $I_1(h)\sim \frac{1}{h^{p/2}}$. We can conclude that the limit for $h\rightarrow \infty$ of $I(h)=I_1(h)+I_2(h)$ has constant phase if $L(x)=0$ for at least a specific $x_j$. If there are no values for which $L(x)=0$, the phase will not converge, this is easy to see in the case we do have at least a critical point $x_j$ with $L(x_j)\neq 0$ and $\det H(x_j)\neq 0$, since in that case the term $e^{ihL(x_j)y^2}$ in $I_2(h)$ will each continue to change phase as $h$ goes to infinity.

Note that if the critical points such that $L\neq0$ have $\det H=0$, we would need to look at higher order asymptotic terms, but, since the number of critical points is finite, we could still look at the highest order, dominant critical points, whose phase is dependent on $e^{ihL(x_j)y^2}$ (\cite{wang} page 483), and this is one of the reasons we need to have, in the most general case, $\bf F$ analytical.

Suppose instead that there are no critical points at all, then the integral in Eq. \ref{PI1} is dominated by the evaluation of some derived phase integral on the boundary of $\mathcal{A}\times \mathcal{B}$, more particularly, it is true that (adapted from \cite{wang}, page 488):
\begin{equation}
\label{wang}
I(h)\sim-\frac{i}{h}\int_{\partial{(A\times B)}}Ge^{ihL(x)y^2}da
\end{equation}
where $\partial{(\mathcal{A}\times \mathcal{B})}$ is the boundary of $\mathcal{A}\times\mathcal{B}$, $da$ is a suitable measure on the boundary, and $G$ is a multiplier function dependent on $L(x)y^2$.

Now, $\mathcal{A}\times \mathcal{B}$ is an hypercube, and a recursive application\footnote{Where successive multiplier functions $G_i$ will depend both on $L^2y^2$ and $G_{i-1}$ (see again \cite{wang} page 487-488).} of the result in Eq. \ref{wang}, to lower and lower dimensional boundaries of its hyperfaces, will reduce the asymptotic evaluation of $I(h)$ to a sum of suitable multiples of evaluations of $e^{ihL(x)y^2}$ at the vertexes of the hypercube. None of these values is independent of $h$, since  we assumed there are no critical points of $L$ on $\mathcal{A} \times \mathcal{B}$, and therefore $L(x)y^2\neq0$ everywhere. This implies that $\lim_{h\rightarrow \infty}\phi(I(h))$ does not exist when there are no critical points on $\mathcal{A}\times \mathcal{B}$.
\end{proof}

\begin{remark}
While the main thrust of this paper is the analysis of irrationality of the evaluation of analytical functions at a point, as it will be clear in the next section, we stress that the result of Theorem 2.1, in its simplicity, offers a potentially powerful new approach for problems in real geometry, and in particular for the solution of problems in real algebraic geometry. To this purpose, Theorem 2.1 would need to be suitably generalized to the case ${\bf F}(x)=0$ has solutions of dimension bigger than zero, along the lines of the results on stationary phase asymptotic approximations on curves described in \cite{wang}, page 459.
\end{remark}

Because of the property proven in Theorem 2.1 that the phase of $I(h)$ in Eq. \ref{PI1} is constant in the limit of $h$ large if and only if there is a solution for the equation ${\bf F}(x)=0$, we call the integrals in Eq. \ref{PI1} {\it geometric phase integrals}. And we call $L(x)$ the {\it geometric Lagrangian} associated to ${\bf F}(x)=0$. We use this terminology in analogy to the Lagrangian functions used in defining path and field integrals \cite{cmf}, trusting that it will be suggestive of further crossfertilization of ideas and methods. In our main setting of the study of irrationality of evaluation of functions, see for example the discussion at the end of Section 4.

\section{Irrationality Tests}

We will now apply this general setting to a more complex case that involves infinitely many critical points, but such that the relative contributions of each can be controlled.

Suppose we want to know whether $F(x_0)=\alpha_0$ is irrational. The system of equations
\begin{equation}
\label{irr}
\begin{split}
%\begin{align}
&F(x)-\alpha=0,\,\,\,x\in[x_0-\delta,\,\,x_0+\delta]&\\
&x-x_0=0,\,\,\,\,\,\,\alpha\in F([x_0-\delta, x_0+\delta])&\\
&\sin\frac{\pi}{m}=\sin\frac{\pi}{n}=0,\,\,\,\,m,n\in(0,\,1]&\\
&\alpha m-n=0&\\
%\end{align}
\end{split}
\end{equation}
has a solution if and only if $\alpha_0$ is a rational number. We can adapt the stationary phase integral analysis performed in Section 2, used to study geometric problems, to be of relevance in this case. We build to this purpose the geometric Lagrangian function:
\begin{equation}
\label{La}
\begin{split}
L(x,\alpha,m,n)=(F(x)-\alpha)^2+(x-x_0)^2+\sin^2\frac{\pi}{m}+
\sin^2\frac{\pi}{n}+(\alpha m-n)^2
\end{split}
\end{equation}
Again, $L(x,\alpha,m,n)=0$ if and only if the previous system has a zero solution, and we may ask whether the limit for $h\rightarrow \infty$ of the phase of the following integral has any relation to the rationality of $F(x_0)=\alpha_0$:
\begin{equation}
\label{PI2}
\begin{split}
I_L(h)=\int_{y\in \mathcal{A}} \int_{\omega\in \Omega_{\delta}} e^{ihL(\omega)y^2} d\omega dy,\,\,0\not \in \mathcal{A}
\end{split}
\end{equation}
where $\omega=(x,\alpha,m,n)$ and we denote by $\Omega_{\delta}$ the tensor product of the domains allowed for each of the components of $\omega$ in Eq. \ref{irr}.

The main complication, with respect to the similar setting in Section 2, is the existence of infinitely many critical points, every time there is at least one point such that $L(\omega)=0$. Consider the partial first derivatives of $L(\omega)$, a critical point of $L(\omega)$ has to satisfy:
\begin{equation}
\label{partials}
%\begin{flalign}
\begin{split}
&\frac{\partial L}{\partial x}=2(F(x)-\alpha)\frac{dF(x)}{dx}+2(x-x_0)=0&\\
&\frac{\partial L}{\partial \alpha}=-2(F(x)-\alpha)+2(\alpha m-n)=0&\\
&\frac{\partial L}{\partial m}=2\sin\frac{\pi}{m}\cos\frac{\pi}{m}(-\frac{\pi}{m^2})+2(\alpha m-n)\alpha=0&\\
&\frac{\partial L}{\partial n}=2\sin\frac{\pi}{n}\cos\frac{\pi}{n}(-\frac{\pi}{n^2})-2(\alpha m-n)=0&\\
\end{split}
%\end{flalign}
\end{equation}
We can see that if $\omega_0=(x_0,\alpha_0,m_0,n_0)$ is a solution of $L(\omega_0)=0$, then it is also a critical point of $L$. However, also $\omega_i=(x_0,\alpha_0,m_i,n_i)$ will be a zero and a critical point of $L$, where $m_i=\frac{m_0}{i}$ and $n_i=\frac{n_0}{i}$, $i$ any integer (this can be seen by simple substitution in $\alpha m-n=0$, assuming $\alpha_0m_0-n_0=0$). Note that all critical points with $L(\omega)=0$ need to have $x=x_0$ and $\alpha=\alpha_0$.

To overcome this proliferation of critical points there are two main issues to consider, the first is that our argument will work only in the limit of the domain approaching the zero for variables  $m,n$. Second, we need to control the decay of the Hessian in the asymptotic expression used to prove Theorem 2.1.

Regarding the first issue, we cut the domain of $m$ and $n$ as $m\in[M,1]$ and $n\in[N,1]$ with $0<M,N<1$ and define the domain
\begin{equation}
\Omega_{\delta}(M,N)=[x_0-\delta,\,\,x_0+\delta]\times F([x_0-\delta, x_0+\delta])\times[M,1]\times[N,1].
\end{equation}

The main conclusion of our analysis can be stated as a theorem:
\begin{theorem}
Let $F(x)$ be an analytical function in the interval $[x_0-\delta,x_0+\delta]$, with $\delta$ sufficiently small, and assume $F'(x_0)\neq 0$. Consider the following phase integral, the restriction of $I_L(h)$ to the domain $\Omega_{\delta}(M,N)$:
\begin{equation}
\begin{split}
I_L(h,M,N)=\int_{y\in \mathcal{A}} \int_{\omega\in \Omega_{\delta}(M,N)} e^{ihL(\omega)y^2} d\omega dy,\,\, 0\not \in \mathcal{A}
\end{split}
\end{equation}
where $L$ is defined in Eq. \ref{La}. Let $\phi(I_L(h,M,N))$ be the phase of $I_L(h,M,N)$.   $F(x_0)=\alpha_0$ is a rational number if and only if the following limit converges:
\begin{equation}
\lim_{M,N\rightarrow 0}\lim_{h\rightarrow \infty} \phi(I_L(h,M,N)).
\end{equation}
\end{theorem}

\begin{proof}
We start our proof with a simple analysis of the dimensionality, in $x$ and $\alpha$, of solutions of the first equation of the systems in Eq. \ref{partials}, that define the critical points. And we eventually prove that for $\delta$ small enough all critical points in $\Omega_{\delta}$ are isolated. To achieve this goal, we note that, for $\delta$ sufficiently small we can control the norm of another function, $(F(x)-\alpha)-F'(x)(x-x_0)$, in $\Omega_{\delta}$, indeed we have
\begin{equation}
%\begin{flalign}
\begin{split}
&|(F(x)-\alpha)-F'(x)(x-x_0)|=|(F(x)-(F(x_0)+\epsilon_1))-(F'(x_0)+\epsilon_2)(x-x_0)|=&\\
&|(F(x)-F(x_0))-F'(x_0)(x-x_0)-\epsilon_1-\epsilon_2(x-x_0)|\leq&\\
&|(F(x)-F(x_0))-F'(x_0)(x-x_0)|+|\epsilon_1|+|\epsilon_2(x-x_0)|\leq&\\
&|(F(x)-F(x_0))-F'(x_0)(x-x_0)|+|\epsilon_1|+|\epsilon_2\delta|\leq&\\
&|\epsilon_3|+|\epsilon_1|+|\epsilon_2\delta|&
\end{split}
%\end{flalign}
\end{equation}
where we used the fact that the derivative of $F(x)$ is well defined and continuous in a neighborhood of $x_0$, and $\epsilon_{t}$, $t=1,2,3$, can be made as small as necessary choosing $\delta$ small enough. But we can interpret this result by saying that the vectors $(F(x)-\alpha,x-x_0)$ and $(1,-F'(x))$ are almost orthogonal for all $(x,\alpha)$ in $\Omega_{\delta}$, with $\delta$ sufficiently small. Now the equation $2(F(x)-\alpha)\frac{dF(x)}{dx}+2(x-x_0)=0$ in Eq. \ref{partials} is equivalent to saying that $(F(x)-\alpha,x-x_0)$ and $(F'(x),1)$ are orthogonal, for some choice of $(x,\alpha)$ in $\Omega_{\delta}$. Together with the previous calculations, this implies, for two dimensional vectors, that $(F'(x),1)$ and $(1,-F'(x))$ should be almost parallel, for such choice of $(x,\alpha)$, instead, these vectors are themselves orthogonal, and we conclude there is no solution of $2(F(x)-\alpha)\frac{dF(x)}{dx}+2(x-x_0)=0$, unless $(F(x)-\alpha,x-x_0)=(0,0)$, in which case $x=x_0$ and $\alpha=F(x_0)$. Note that this argument depends on the assumption $F'(x_0)\neq0$ otherwise we would not be able to infer $\alpha=F(x_0)$ from $x=x_0$, in the first equation of Eq. \ref{partials}.

We deduce moreover, from the whole set of equations in Eq. \ref{partials}, that critical points with $x=x_0$ and $\alpha=F(x_0)$, if they exists, are bound to have $\alpha m-n=0$, $2\sin\frac{\pi}{m}\cos\frac{\pi}{m}(-\frac{\pi}{m^2})=0$, and $2\sin\frac{\pi}{n}\cos\frac{\pi}{n}(-\frac{\pi}{n^2})=0$. Therefore they are all isolated points, in finite number on all compacts $\Omega_{\delta}(M,N)$ and they either satisfy $\sin\frac{\pi}{m}=0$ and $\sin\frac{\pi}{n}=0$ (and therefore $L(\omega)=0$), or they are such that $\cos\frac{\pi}{m}=0$ and/or $\cos\frac{\pi}{n}=0$. Since critical points are isolated and finitely many in $\Omega_{\delta}(M,N)$, for any $0<M,N<1$, we are in the position of applying Theorem 2.1 in the rest of the proof.

The proof of the theorem then relies on the following estimate: suppose $\alpha_0$ is rational and that $m_0,n_0$ are the largest values such that $L(x_0,\alpha_0,m_0,n_0)=0$, then
\begin{equation}
\label{estimate}
\begin{split}
\det H(x_0,\alpha_0,m_i,n_i)\thicksim C\frac{i^8}{m_0^8}
\end{split}
\end{equation}
in the limit of $i$ that goes to infinity, where $m_i=\frac{m_0}{i}$, $n_i=\frac{n_0}{i}$, $i$ positive integer and $C$ is a positive number bigger than $1$. Indeed, remembering that, for critical points $\omega_i=(x_0,\alpha_0,m_i,n_i)$ with $L(\omega_i)=0$, we have $\alpha_0 m_i-n_i=0$, $\sin\frac{\pi}{m_i}=0$, $\sin\frac{\pi}{n_i}=0$ (and therefore $\cos\frac{\pi}{n_i}=1$, $\cos\frac{\pi}{m_i}=1$), we can write the Hessian matrix of $L(\omega)$ evaluated at such critical points as:

\begin{equation}
 H(\omega_i) = \left( \begin{array}{cccc}
2F'(x_0)^2+2 & -2F'(x_0) & 0 & 0 \\
-2F'(x_0) & 2+2m_i^2 & 2\alpha_0 m_i & -2m_i \\
0 & 2\alpha_0 m_i & 2\frac{\pi^2}{m_i^4}+2\alpha_0^2 & -2\alpha_0 \\
0 & -2 m_i & -2\alpha_0 & 2\frac{\pi^2}{n_i^4}+2\end{array} \right).
\end{equation}
Using again the fact that, for these critical points, $\alpha_0 m_i=n_i$, the evaluation of the determinant gives:
\begin{equation}
%\begin{flalign}
\begin{split}
&\det H(\omega_i)=(4F'(x_0)^2+4m_i^4 +4)\big( 4( \frac{\pi^2}{m_i^4}+\alpha_0^2)(\frac{\pi^2}{\alpha_0^4 m_i^4}+1)-4\alpha_0^2\big)&\\
&+4(F'(x_0)^2+2)\alpha_0^2m_i^2(-\frac{\pi^2}{\alpha_0^4m_i^4}-1+8)-16(F'(x_0)^2+1)m_i^2(\frac{\pi^2}{m_i^4}+\alpha_0^2).&
\end{split}
%\end{flalign}
\end{equation}
where we did not fully simplify the expression to leave the reader with a sense of its structure.
Recalling $m_i=\frac{m_0}{i}$ with $i=1,2,3...$, if we let $i\rightarrow \infty$ (i.e. $m_i\rightarrow 0$), the leading term of the determinant will be:
\begin{equation}
\det H(\omega_0)\thicksim 16\frac{\pi^2}{m_i^4}\frac{\pi^2}{\alpha_0^4 m_i^4}=\frac{16\pi^4}{\alpha_0}\frac{i^8}{m_0^8}
\end{equation}
which is the estimate in Eq. \ref{estimate}, with $C=\frac{16\pi^4}{\alpha_0}$. This being the case, we can be assured that there is a $i_T$ such that for $i>i_T$ the Hessian $H(x_0,\alpha_0,m_i,n_i)$ has nonzero (positive) determinant, and therefore the quadratic asymptotic approximation used in Theorem 2.1 holds for all $i>i_T$.

Also, note that, for $i<i_T$ any critical point such that $H(x_0,\alpha_0,m_i,n_i)=0$ will depend from $h$, in the asymptotic expansion, as $\frac{1}{h^{j+2}}$ for some integer $j>0$ that depends from the order of the zero, while all critical points with $H(x_0,\alpha_0,m_i,n_i)\neq 0$ depend from $h$ as $\frac{1}{h^{2}}$ (\cite{wang}, page 480). This implies that we can neglect critical points that have Hessian equal to zero, in the limit of $h\rightarrow \infty$, since the asymptotic relation in Eq. \ref{estimate} assures us that there are infinitely many dominant critical points with non-zero determinant of the Hessian in $\Omega_{\delta}$, and therefore at least one of them for $M,N$ sufficiently small. Therefore we have:
\begin{equation}
\begin{split}
\lim_{M,N\rightarrow 0}\lim_{h\rightarrow \infty} \phi(I_L(h,M,N))=\\
\lim_{M,N\rightarrow 0}\lim_{h\rightarrow \infty} \int_{y\in A} \sum\limits_{\substack{L(\omega_i)=0\\\det H(\omega_i)\neq 0\\\omega_i\in\Omega_{\delta}(M,N)}} (\frac{2\pi}{h})^{2} \frac{1}{y^4 (\det H(\omega_i))^{1/2}}e^{i\frac{\pi}{4}\sigma_i}
\end{split}
\end{equation}
where we have used the results from Theorem 2.1, the fact that $p=4$, and neglected already the (finitely many) critical point for which $L(\omega)\neq0$, or those for which $L(\omega_i)=0$ and $\det H(\omega_i)=0$.

Consider now the partial sums:
\begin{equation}
\begin{split}
\theta_{M,N}=\sum\limits_{\substack{L(\omega_i)=0\\\det H(\omega_i)\neq 0\\\omega_i\in\Omega_{\delta}(M,N)}}  \frac{({2\pi})^{2}}{ (\det H(\omega_i))^{1/2}}e^{i\frac{\pi}{4}\sigma_i}
\end{split}
\end{equation}
then
\begin{equation}
\begin{split}
\lim_{M,N\rightarrow 0}\lim_{h\rightarrow \infty} \phi(I_L(h,M,N))=\lim_{M,N\rightarrow 0}\lim_{h\rightarrow \infty} \phi\big{(}\int_{y\in A} \frac{1}{h^2} \frac{1}{y^4}\theta_{M,N} dy)\big{)}=\\
\lim_{M,N\rightarrow 0}\lim_{h\rightarrow \infty}\phi \big{(}\frac{1}{h^2}S \theta_{M,N}\big{)}=
\lim_{M,N\rightarrow 0}\phi(\theta_{M,N})
\end{split}
\end{equation}
where $S=\int_{y\in \mathcal{A}} \frac{1}{y^4}dy$. Now, because of the relation $\det H(x_0,\alpha_0,m_i,n_i)\thicksim C\frac{i^8}{m_0^8}$, for $i\rightarrow \infty$  we can argue that
the following series converges:
\begin{equation}
\label{theta}
\begin{split}
\theta=\sum\limits_{\substack{L(\omega_i)=0\\\det H(\omega_i)\neq 0\\\omega_i\in\Omega_{\delta}}}
\frac{({2\pi})^{2}}{ (\det H(\omega_i))^{1/2}}e^{i\frac{\pi}{4}\sigma_i}
\end{split}
\end{equation}

Indeed, the convergence of the this series can be reduced to the convergence of its absolute value
\begin{equation}
\label{abs}
\sum\limits_{\substack{L(\omega_i)=0\\\det H(\omega_i)\neq 0\\\omega_i\in\Omega_{\delta}}} \frac{({2\pi})^{2}}{ (\det H(\omega_i))^{1/2}}
\end{equation}
and, by comparison with the convergent series
$\sum_{i} \frac{1}{ i^4}$, the limit comparison test of convergence gives us:
\begin{equation}
\begin{split}
\lim_{i\rightarrow \infty} \frac{({2\pi})^{2}}{ (\det H(\omega_i))^{1/2}}\big{/} \frac{1}{ i^4}=\lim_{i\rightarrow \infty}\frac{({2\pi})^{2}}{ \sqrt{C}(i^8/m_0^8)^{1/2}}\big{/} \frac{1}{ i^4}=
({2\pi})^{2} m_0^4/\sqrt{C}
\end{split}
\end{equation}

Since the limit of the quotient above is nonzero, the series in Eq. \ref{abs} converges, and $\theta$ in Eq. \ref{theta} is well defined. The convergence of the series defining $\theta$ allows us one final limiting argument, i.e.,
\begin{equation}
\begin{split}
\lim_{M,N\rightarrow 0}\lim_{h\rightarrow \infty} \phi(I_L(h,M,N))=
\lim_{M,N\rightarrow 0}\phi(\theta_{M,N})=\phi(\theta).
\end{split}
\end{equation}
And this last equality completes the proof of the Theorem.
\end{proof}

\begin{remark}
\label{comments}
The convergence of the series defining $\theta$ in Eq. \ref{theta} is intimately related to the estimate in Eq. \ref{estimate}. The existence of this estimate depend on the fact that we use the equations $\sin\frac{\pi}{m}=0$, $\sin\frac{\pi}{n}=0$, on a bounded domain, to force rationality of $F(x)=\alpha$ (via the additional equation $\alpha m-n=0$). Such convergence would not hold if rationality was enforced via the equations $\sin{\pi}{m}=0$, $\sin{\pi}{n}=0$ on an unbounded domain. Note also that the phase integral in Eq. \ref{PI2} depends {\it functionally} on $F(x)$, so that the local behavior of $F(x)$ for $x\sim x_0$ becomes relevant for the irrationality of $F(x_0)=\alpha_0$.
\end{remark}

\begin{remark}
Our choice of the dependence of the geometric Lagrangians from variable $y$ is not the only one that would establish the results in Theorems 2.1 and 3.1, even though it is probably the simplest. Alternatively, one could look at the geometric Lagrangian $L(\omega)\exp(y)+y^3$ whose critical points are only those associated to $L(\omega)=0$, removing the necessity of the careful estimate of the contribution of critical points with $L(\omega)\neq 0$. However, this more complicated  geometric Lagrangian leads always to degenerate critical points in the stationary phase asymptotic approximation and therefore to a more intricate proof of the two Theorems.
\end{remark}

There are several problems that could benefit from the application of Theorem 3.1, however, we formally write only one such application for a number, the gamma constant $\gamma$, whose irrationality is not known. The Digamma function $\Psi$ can be used to define the Euler-Mascheroni $\gamma$ constant as $\Psi(1)=-\gamma$, and since $\Psi(x)$ is analytical at $x=1$, with $\Psi'(1)=\frac{\pi^2}{6}\neq0$ we can state the following Corollary to Theorem 3.1, where we assume $\delta$ has already been chosen sufficiently small:

\begin{corol}
Consider the geometric Lagrangian associated to the Digamma function $\Psi$:
\begin{equation}
\begin{split}
L(x,\alpha,m,n)=(\Psi(x)-\alpha)^2+(x-1)^2+\sin^2\frac{\pi}{m}+
\sin^2\frac{\pi}{n}+\\
(\alpha m-n)^2.
\end{split}
\end{equation}
The Euler-Mascheroni constant $\gamma$ is rational if and only if the following limit converges
\begin{equation}
\lim_{M,N\rightarrow 0}\lim_{h\rightarrow \infty} \phi(I_L(h,M,N)).
\end{equation}
\end{corol}

\section{Further Developments}

The method we outlined in Section 3 is not restricted to the study of rationality of functions evaluated at one point. Suppose we are interested in the problem of finding whether $F(x)=0$, $x\in \mathcal{B}$ has rational solutions, with $\mathcal{B}$ a compact domain. The method described in Section 3 will apply, using the Geometric Lagrangian:
\begin{equation}
\label{LaDioph}
\begin{split}
L(x,m,n)=F(x)^2+\sum_{i=1}^p\sin^2\frac{\pi}{m_i}+
\sin^2\frac{\pi}{n_i}+(x_i m_i-n_i)^2
\end{split}
\end{equation}
A full adaptation of Theorem 3.1 to this case requires a careful evaluation of convergence of multiple series associated to critical points with $F(x)=0$, and this is problematic if there are infinitely many rational solutions of $F(x)=0$ in $\mathcal B$. However, the proof of Theorem 3.1 directly applies when $F(x)=0$ has finitely many rational solutions on $\mathcal B$, just by considering the finitely many convergent series of the type in Eq. \ref{theta}, associated to each rational solution, if any. This implies that a slightly modified version of Theorem 3.1 holds for the study of rational curves of genus bigger than 1 since such curves always have at most finitely many rational points \cite{MF}.

Also the problem of determining whether $F(x)=0$ has algebraic solutions of degree $K$ can also be stated in terms of phase integrals on geometric Lagrangians of the type:
   \begin{equation}
\label{LaAlgebK}
\begin{split}
L(x,a,m,n)=F(x)^2+\sum_{i=1}^p g_i(x_i)^2+\\
\sum_{i=1}^p \sum_{j=1}^K\sin^2\frac{\pi}{m_{ij}}+
\sin^2\frac{\pi}{n_{ij}}+(a_{ij} m_{ij}-n_{ij})^2
\end{split}
\end{equation}
where $g_i(x_i)$ are polynomials of degree at most $K$, $a_{ij}$ are their (rational) coefficients and we denote by $a$ the vector of all $a_{ij}$. It is not clear however whether the phase integral method can be adapted to discriminate algebraic numbers, of any degree $K$, from transcendental numbers.

Diophantine equations could be similarly approached. If we are interested in the existence of integer solutions of $F(x)=0$ on a bounded domain $\mathcal B \subset \mathbb{R}^p$, the following geometric Lagrangian would be suitable:
  \begin{equation}
\label{LaAlgebK}
\begin{split}
L(x,m)=F(x)^2+\sum_{i=1}^p \sin^2\frac{\pi}{m_{i}}+(x_{i}m_{i}-1)^2
\end{split}
\end{equation}

We have already stressed in Remark \ref{comments} the importance of keeping a functional dependence of the phase integral from $F(x)$. This dependence can be used to say a little more about the structure of the geometric Lagrangians defined so far, as they can all be split into three components. Let's focus for simplicity on the geometric Lagrangian for the irrationality test in Section 3 (the same arguments extend easily to the Lagrangians sketched in this section).

The geometric Lagrangian
$L(x,\alpha,m,n)=(F(x)-\alpha)^2+(x-x_0)^2+\sin^2\frac{\pi}{m}+
\sin^2\frac{\pi}{n}+(\alpha m-n)^2$ can be written as $L(x,\alpha,m,n)=L_1(x,\alpha)+L_2(m,n)+L_3(\alpha,m,n)$, now the functions $L_1(x,\alpha)=(F(x)-\alpha)^2+(x-x_0)^2$ and $L_2(m,n)=\sin^2\frac{\pi}{m}+
\sin^2\frac{\pi}{n}$ can be seen as two distinct Lagrangians, each leading to geometric phase integrals whose phase is always convergent, while $L_3(\alpha,m,n)=(\alpha m-n)^2$ can be seen as a ``coupling Lagrangian'' that provides the interaction between the first two Lagrangians.

This discussion is inspired by the language of quantum field theory, where interaction among free fields is often mediated by only some of the terms in the associated Lagrangian (\cite{cmf}, chapter 5). If we push this analogy even further, we can say that, for any small coupling parameter $\beta>0$ the Lagrangian $L_{\beta}(x,\alpha,m,n)=(F(x)-\alpha)^2+(x-x_0)^2+\sin^2\frac{\pi}{m}+
\sin^2\frac{\pi}{n}+\beta(\alpha m-n)^2=L_1(x,\alpha)+L_2(m,n)+L_{3,\beta}(\alpha,m,n)$ is just as suitable to study the irrationality of $F(x_0)=\alpha_0$. For $\beta$ very small, this modification allows to expand the phase integral associated to $L_{\beta}$ in terms of powers of $L_{3,\beta}$,
since, for $\beta$ sufficiently small, $L_{3,\beta}=\beta L_3=\beta(\alpha m-n)^2$ will also be small on $\Omega_{\delta}$. More particularly, for $\beta$ very small, and under the conditions of Theorem 3.1, we have the following equalities:
\begin{equation}
\label{beta}
%\begin{flalign}
\begin{split}
&\int_{y\in \mathcal{A}} \int_{\omega\in \Omega_{\delta}(M,N)} e^{ihL_{\beta}(\omega)y^2} d\omega dy=&\\
&\int_{y\in \mathcal{A}} \int_{\omega\in \Omega_{\delta}(M,N)} e^{ih(L_1(\omega)+L_2(\omega))y^2} e^{ihL_{3,\beta}(\omega)y^2}d\omega dy=&\\
&\int_{y\in \mathcal{A}} \int_{\omega\in \Omega_{\delta}(M,N)} e^{ih(L_1(\omega)+L_2(\omega))y^2} \sum_{j=0}^{\infty} \frac{(ihL_{3,\beta}(\omega)y^2)^j}{j!}d\omega dy.&\\
\end{split}
%\end{flalign}
\end{equation}
And, being mindful of the contrasting tension between the requirement $h\rightarrow \infty$ and $\beta\rightarrow 0$, the study of the convergence of the phase of the first integral in Eq. \ref{beta} could be replaced by the study of the convergence of the phase of the following series of integrals, potentially allowing perturbative and renormalization methods to be relevant here:
\begin{equation}
\sum_{j=0}^{\infty} \frac{(ih\beta)^j}{j!} \int_{y\in \mathcal{A}} \int_{\omega\in \Omega_{\delta}(M,N)}L_{3}(\omega)^jy^{2j} e^{ih(L_1(\omega)+L_2(\omega))y^2}d\omega dy.
\end{equation}
Not only, it is possible to construct an entire family of Lagrangians $\{L_{\beta}\}$, and study the structure of the ``flow'' of the associated phase integrals as $\beta\rightarrow 0$. Note that, for any $\beta\neq 0$, if $F(x_0)=\alpha_0$ is irrational, there will be no phase convergence of $I(h,M,N)$ as defined in Theorem 3.1, but for $\beta=0$ there will always be phase convergence since the Lagrangians $L_1$ and $L_2$ will be decoupled in that case, and there will always be solutions to the associated geometric problem. So the problem of irrationality of $F(x_0)=\alpha_0$ can also be approached as an abrupt qualitative transition, at $\beta=0$, of the structure of the family of phase integrals associated to the Lagrangians $\{L_{\beta}\}$, again enriching irrationality problems with the methodologies that have been developed to study phase transitions in physics.

While this heuristic discussion is brief and very informal, it is included in the paper to be suggestive of the significant conceptual shift that is possible, by using Theorems 2.1 and 3.1 as a starting point for a renewed study of irrationality and real geometry.

\end{document}